\documentclass[preprint, noinfoline,addressatend,imslayout,a4paper]{imsart}

\usepackage[authoryear]{natbib}
\usepackage{hyperref}

\usepackage{amsthm,amsmath,amsfonts,amssymb}
\usepackage{enumerate}

\usepackage{bm,accents}
\usepackage{graphicx}

\startlocaldefs

\newtheorem{thm}{Theorem}
\newtheorem{cor}[thm]{Corollary}
\newtheorem{lem}[thm]{Lemma}

\setattribute{tablecaption}{shape}{}

\setattribute{tablename} {skip}{.~}

\newcommand{\indi}[1]{\bm{1}{\left\{#1\right\}}}

\newcommand{\R}{\mathbb{R}}
\newcommand{\eps}{\varepsilon}

\newcommand{\argmin}{\operatornamewithlimits{argmin}}
\newcommand{\asarrow}[0]{\stackrel{\mbox{\scriptsize\rm a.s.}}
{\rightarrow}}

\newcommand{\darrow}[0]{\rightsquigarrow}

\newcommand{\F}[0]{\mathcal{F}}

\newcommand{\Z}[0]{\mathcal{Z}}

\newcommand{\sn}[0]{\sqrt{n}}

\newcommand{\Ghat}[0]{\hat{G}}

\newcommand{\ep}{\mathbb{P}}

\newcommand{\hatG}{\hat{G}_n}

\newcommand{\Gtau}{ {G_{\tau}} }
\newcommand{\Stau}{ {S_{\tau}} }
\newcommand{\Htau}{ {H_{\tau}} }
\newcommand{\Gtauh}{ {\hat G_{\tau}} }
\newcommand{\Stauh}{ {\hat S_{\tau}} }
\newcommand{\Htauh}{ {\hat H_{\tau}} }

\bibpunct{(}{)}{;}{a}{,}{,}
\def\spacingset#1{\renewcommand{\baselinestretch}%
{#1}\small\normalsize} \spacingset{1.6}

\begin{document}

\begin{frontmatter}


\title{Hoeffding-Type and Bernstein-Type Inequalities for~Right~Censored~Data} 
\runtitle{Concentration Inequalities for Right Censored Data}
\author{\fnms{Yair} \snm{Goldberg}\ead[label=e1]{yairgo@technion.ac.il}}
%
%
\affiliation{The Faculty of Industrial Engineering and Management,\\ Technion - Israel Institute of Technology}
%
%
\runauthor{Goldberg}
\begin{abstract}
We present Hoeffding-type and Bernstein-type inequalities for right-censored data. The inequalities bound the difference between an inverse of the probability of censoring weighting (IPCW) estimator and its expectation. We first discuss the asymptotic properties of the estimator and provide conditions for its efficiency. We present standard, data dependent, and uniform Hoeffding-type inequalities. We then present a Bernstein-type inequality. Finally, we show how to apply these inequalities in an empirical risk minimization setting.
\end{abstract}
%
%
\end{frontmatter}

\section{Introduction}\label{sec:intro}
%
%
%
%
%

%

Concentration inequalities provide probability bounds on how an empirical quantity of interest deviates from its expectation. Unlike asymptotic results, such as the law of large numbers and the central limit theorem, concentration inequalities are applicable even to small-size samples. Typically, the empirical quantity of interest is the sample mean, where this mean is taken over realizations of some function. Examples for concentration inequalities include Chebyshev's inequality, Hoeffding's inequality, and Bernstein's inequality (see \citealp{Boucheron2004Concentration}, and \citealp{Chung06}, for surveys). In many cases, one is interested in a uniform concentration inequality, namely a bound on the deviation of the sample mean from its expectation over a set of functions rather than over a single function.

Concentration inequalities play an important role in both the design and analysis of empirical risk minimization (ERM) techniques \citep{Vapnik99,Koltchinskii2011}. ERM techniques are statistical techniques which find the parameter of interest by minimizing some (possibly penalized) risk function with respect to the empirical measure. Examples of ERM techniques include the maximum likelihood estimation as well as many machine-learning algorithms such as boosting and support vector machines (SVM). Using uniform concentration inequalities, one can bound the estimation error, i.e., the difference in risk between the empirical risk minimizer and the infimum within the approximation space. Moreover, one can use such bounds to choose the approximation space itself; a method that is utilized, for example, by model selection techniques \citep{Vapnik99,Koltchinskii2011}.

Applying concentration inequalities to survival data is challenging. In survival data, the quantity of interest is a function that involves the failure time. This failure time is usually subject to right censoring, which arises, for example, when a medical study ends before the failure event occurs, or when patients drop out of the study \citep{FH}. In these cases, the failure time is not known, and instead a lower bound on the failure time is given. Existing concentration inequalities cannot be applied to censored data since even the sample mean cannot be calculated due to the censoring. As a consequence, little of machine learning theory is applicable to survival data. Indeed, while some machine learning algorithms were proposed for right-censored data, including \citet{NN98}, \citet{Ripley}, \citet{Johnson04}, \citet{SVR07}, and \citet{SVQR09}, to the best of our knowledge, the theoretical properties of these algorithms have never been studied. Machine learning algorithms with some theoretical results include \citet{Ishwaran2010} for random survival trees, which requires the assumption that the feature space is discrete and finite; \citet{Eleurti_2012} and \citet{GK_SVR} in the context of SVM; and \citet{GK_CQL11} in the context of multistage decision problems.

Since concentration inequalities bound the probability that an estimator, such as the sample mean, is far from its expectation, we need to consider which estimator is appropriate. The standard estimator is not well defined since for censored data, the actual failure time is unknown. Moreover, the naive estimator, that ignores the censored observations, is biased, since longer survival times are more likely to be censored. An alternative to the naive estimator is to take a weighted average of the uncensored observations, according to the inverse of the probability of censoring weighting (IPCW) \citep{Robins94}. We discuss the asymptotic properties of this estimator and provide conditions for its efficiency.

In this paper we develop
concentration inequalities that can be applied to right-censored data, and show how to use them to develop learning algorithms for this type of data. More specifically, we develop Hoeffding-type inequalities that bound the difference between the IPCW estimator and its expectation. These inequalities include inequalities for a single function, data-dependent inequalities \citep[see][for details]{Maurer2009EmpiricalBernstein}, and inequalities for classes of functions. We then present a Bernstein-type inequality. The concentration inequalities that we present here can be used as tools to further develop machine learning theory for right-censored data.

The paper is organized as follows. In
Section~\ref{sec:IPCW} we discuss the IPCW estimator and study its asymptotic properties. The concentration inequalities are presented in Section~\ref{sec:single}. Application of the inequalities to empirical risk minimization is presented in Section~\ref{sec:functions}. Concluding
remarks appear in Section~\ref{sec:summary}. All proofs are deferred to the  Appendix.

\section{Inverse Probability Censoring Weighting Estimator}\label{sec:IPCW}
We start by presenting the standard Hoeffding and Bernstein inequalities. We then explain why no trivial adaptation of these  inequalities to right-censored data exists.

Let $\{(T_1,Z_1),\ldots, (T_n,Z_n)\}$ be a set of $n$ i.i.d.\ random pairs, where $T_i$ gets values in the segment $[0,\tau]$ for some constant $\tau$ greater than zero, and $Z_i$ is some random vector of dimension $d$ that gets it values in $\Z\subset\R^d$. Let
\begin{align}\label{eq:f}
f:[0,\tau]\times\Z\mapsto \R\,;\quad \|f\|_{\infty}\leq M
\end{align}
be a measurable function where $M>0$ is some constant. Let $\bar{\mu}_n(f)=n^{-1}\sum_{i=1}^n f(T_i,Z_i)$ be the sample mean of $f(T,Z)$ and let $\mu(f)\equiv P[f(T,Z)]$ be the expectation of $f(T,Z)$. Hoeffding's inequality states that for all $\eta>0$,
\begin{align}\label{eq:Hoeffding}
P\left( |\bar{\mu}_n(f)-\mu(f)|\geq M\sqrt{\frac{2\eta}{n}}\right)\leq 2e^{-\eta}\,.
\end{align}
Bernstein's inequality states that for all $\eta>0$,
\begin{align}\label{eq:Bernstein}
P\left( |\bar{\mu}_n(f)-\mu(f)|\geq \sqrt{\frac{2\mathrm{Var}(f)\eta}{n}}+\frac{2M\eta}{3n}\right)\leq 2e^{-\eta}
\end{align}
\citep[see, for example,][]{Boucheron2004Concentration}.

The inequalities above provide bounds on the probability that the difference between the natural estimator of the expectation of $f(T,Z)$ and the expectation itself is larger than some constant normalized by root-$n$.

Assume now that the data is subject to right censoring, and thus consist of $n$ independent and
identically-distributed random triplets $\{(U_1,\delta_1,Z_1),\allowbreak \ldots,(U_n,\delta_n,Z_n)\}$. The random variable $U$ is the observed time defined by $U=\min\{T,C\}$; where $T$ is the failure time, and $C$ is the censoring time. The indicator $\delta=\indi{T\leq C}$ is the failure indicator, where $\indi{A}$ is $1$ if $A$ is true and is $0$ otherwise, i.e., $\delta=1$ whenever a failure is observed.
Let $S(t)=P(T>t)$ be the survival function of $T$,
$G(t)=P(C> t)$ be the survival function of $C$, and $H(t)=P(U>t)$ be the survival function of $U$.

As discussed in Section~\ref{sec:intro}, estimating the expectation of right-censored data is challenging. The standard estimator is not applicable and the naive estimator that ignores the censored observations \begin{align*}
\frac{1}{\# \{\delta_i=1\}}\sum_{i=1}^n \delta_i  f(T_i,Z_i)
\end{align*}
is biased, since longer survival times are more likely to be censored. An alternative to the biased estimator above is to take a weighted average of the uncensored observations, according to the inverse of the probability of censoring weighting (IPCW) \citep{Robins94}.

Let $\hatG (t)$ be the Kaplan-Meier estimator for $G$. Recall that $\hatG $ is a consistent and
efficient estimator for the survival function $G$
\citep[][Chapter~4.3]{Kosorok08}. An estimator of the expectation of
$f(T,Z)$ based on the inverse probability censoring weighting (IPCW)
\citep{Robins94} is given by
\begin{align}\label{eq:estimator}
\hat{\mu}_n(f) \equiv \ep_n \frac{\delta f(T,Z)}{    \hatG(T-)}\equiv \frac1n \sum_{i=1}^{n}\frac{\delta_i f(T_i,Z_i)}{    \hatG(T_i-)}\,,
\end{align}
where $\ep_n$ is the empirical measure, and $\hatG(t_0-)=\lim_{t\rightarrow t_0, t<t_0}\hatG(t)$.

We would like to show that estimator~\eqref{eq:estimator} is a ``good" estimator. We first prove that this estimator is asymptotically unbiased, and weakly converges to a normal random variable. Moreover, when $f$ is a function only of  $T$,  we show that it is an efficient estimator. In order to show the last result we first prove that the estimator~\eqref{eq:estimator} has a representation as a Kaplan-Meier functional~\citep{Schick1988}. While the fact that the Kaplan-Meier estimator itself has a representation as an IPCW estimator \citep{zhao_consistent_1997,Satten01} is known\footnote{I would like to thank A.A. Tsiatis for pointing out this identity and some of its implications.}, our result generalizes it to a wider class of estimators.

We need the following assumptions:
\begin{enumerate}
	\renewcommand{\labelenumi}{(A\arabic{enumi})}
	\item The pair $(T,Z)$ is independent of the censoring time $C$.\label{as:independentCT}
	
	\item $T$ takes values in the segment $[0,\tau]$ for some finite $\tau>0$, and $H(\tau-)>0$.\label{as:positiveRisk}
	
	\item The probability that $T=C $ is zero. \label{as:no_simultanious_failure}
	
\end{enumerate}
The first assumption states that the censoring arises completely at random. This assumption is reasonable, for example, when the censoring is administrative. The second assumption assures that there is a positive probability of observing failure time up to time~$\tau$. Note that the existence of such a~$\tau$ is typical since most studies have a finite time period of observation. Moreover, if some random variable~$T^*$ gets values greater than $\tau$, one can always look at $T=\min\{T^*,\tau\}$ \citep[see, for example,][Section~2, for discussion]{zhao_consistent_1997}. Note that having one observation at time $\tau$ verifies Assumption~(A\ref{as:positiveRisk}). Assumption~(A\ref{as:no_simultanious_failure}) states that the probability of simultaneous failure and censoring event is zero. This assumption holds, for example, if either $T$ or $C$ are absolutely continuous on the segment $[0,\tau)$. Even when the distributions of both $T$ and $C$ have finite number of atoms, the assumption holds for $C+\varepsilon$ of some small $\varepsilon>0$.

Denote $\Stau\equiv S(\tau-) =P(T\geq \tau)$, $\Gtau\equiv G(\tau-)=P(C\geq \tau)$,  and $\Htau \equiv H(\tau-)=P( U\geq \tau)$. Note that by Assumption (A\ref{as:independentCT}), $S_{\tau} \cdot G_{\tau}= H_{\tau}>0$.  
Denote by $\Stauh$ and $\Gtauh$ the Kaplan-Meier estimators of $\Stau$ and $\Gtau$; and by $\Htauh$ is the estimator of $\Htau$ based on the cumulative distribution estimator of $U$.

\begin{thm}\label{thm:asymptotic}
	Assume (A\ref{as:independentCT})--(A\ref{as:no_simultanious_failure}) and let $P[f(T,Z)^2]<\infty$. Then
	\begin{enumerate}
		\item $\displaystyle 
		\hat{\mu}_n \asarrow \mu(f) $.
		
		\item $\displaystyle 
		\sn\left(\hat{\mu}_n(f)- \mu(f) \right)~\darrow N\left(0,\sigma_f^2\right)\,,$
		where 
		\begin{align}\label{eq:sigma_f}
		\sigma_f^2\equiv   \mathrm{Var}(f(T,Z))+P\left[\int_0^\tau\left(f(s,Z)-\frac{P\left(f(T,Z) \indi{T\geq s}\right)}{S(s-)}\right)^2\indi{T\geq s}\frac{d\Lambda^C(s)}{G(s)}\right]\,,
		\end{align}
		and where $\Lambda^C$ is the cumulative hazard function of the censoring variable $C$.
		
		\item When $f $ is a function only of $T$,
		\begin{align*}
		\hat{\mu}_n(f) = \int_0^\tau f(t) d\hat F_n(t)\,,
		\end{align*}
		where $\hat F_n$  is the Kaplan-Meier estimator of $F\equiv 1-S$. Moreover, $\hat{\mu}_n(f)$ is an efficient estimator for $\mu(f)$.
		
	\end{enumerate}

\end{thm}
See proof in Appendix~\ref{subsec:asymptotic}


\section{Concentration Inequalities}\label{sec:single}
In the previous section we showed that the IPCW estimator is asymptotically normal. In this section we would like to present non-asymptotical bounds on the distance of the estimator from its expectation. In Sections~\ref{sec:simple},~\ref{sec:empirical}, and~\ref{sec:uniform}, we discuss simple, empirical, and uniform Hoeffding-type inequalities, respectively. In Section~\ref{sec:Bernstein}, we use these inequalities to prove a Bernstein-type inequality. A discussion appears in Section~\ref{subsec:discussion}.

\subsection{Simple Hoeffding-Type Inequality}\label{sec:simple}
We start with the following Hoeffeding-type inequality for right-censored observations:
\begin{thm}\label{thm:hoeffding}
	Assume (A\ref{as:independentCT})--(A\ref{as:no_simultanious_failure}) and that $\|f\|_{\infty}\leq  M$. Then for any $n\geq 1$ and $\eta>0$ we have
	\begin{align*}
	P\left( \frac{\sn  \Htau\Gtauh}{M}\left|\hat{\mu}_n(f)-\mu(f) \right|\geq 3\sqrt{\frac{\eta}{2}} +\frac{D_o}{2}\right)\leq \frac{9}{2}e^{-\eta}\,,
	\end{align*}
	where $D_o$ is some universal constant.
\end{thm}
See proof in Appendix~\ref{subsec:proof_hoeffding}.

Note that this result involves the two constants $D_o$ and $\Htau$, as well as the empirical quantity $\Gtauh$. The constant $D_o$ appears in the Dvoretzky-Kiefer-Wolfowitz-type
inequality for Kaplan-Meier estimator (DKW-KM for short) \citep[][Theorem~2]{Bitouze99} which we use in the proof below. An upper bound on $D_o$ is given by \citet{Wellner07}, although this bound is large and simulations suggest that $D_o$ is much smaller. Moreover, $D_o$ does not depend on the distributions of $T$ and $C$. The constant $\Htau$ does depend on the probability distributions, but since  $\Htau=E[\indi{U\geq \tau}]$ and the variable $\indi{U\geq \tau}$ is fully observed, $\Htau$ can be approximated well. The quantity $\Gtauh$ is random although it is known for every given sample. Replacing this quantity with a distribution-dependent quantity can be done using the multiplicative Chernoff bound~\citep{Hagerup1990}.  
\begin{cor}\label{thm:hoeffding_cor}
	Assume (A\ref{as:independentCT})--(A\ref{as:no_simultanious_failure}) and that $\|f\|_{\infty}\leq  M$. Then for any $n\geq 1$ and $\eta>0$ we have
	\begin{align*}
	P\left( \frac{\sn  \Htau^2}{M}\left|\hat{\mu}_n(f)-\mu(f) \right|\geq  3\sqrt{\frac{\eta}{2}} +\frac{D_o}{2}+2\right)\leq \frac{9}{2}e^{-\eta}+e^{-\Htau/(3n)}\,.
	\end{align*}
\end{cor}
See proof in Appendix~\ref{subsec:proof_hoeffding_cor}.

In the following section we present an empirical version of these inequalities that does not involve constants that are distribution dependent.

\subsection{Empirical Hoeffding-Type Inequality}\label{sec:empirical}
Before we discuss the empirical Hoeffding-type inequality, we need to discuss an empirical version of the DKW-KM inequality. We start by presenting a simplified version of the DKW-KM inequality \citep[see][Theorem~2 for details]{Bitouze99}. For every $n\geq 1$ and $\eps>0$, \begin{align}\label{eq:DKW}
P(
\sn  \Stau  \|\hatG-G\|_{\infty} \geq\eps)&\leq\frac{5}{2}e^{-2\eps^2+D_o\eps}\,.
\end{align}

Some algebraic manipulations (see details in Appendix~\ref{subsec:proof_of_eq})
yield the following version of the DKW-KM inequality.
\begin{align}\label{eq:KM_exp_bound}
P\left(\sn\Stau \|\hatG-G\|_{\infty}\geq\sqrt{\frac{\eta}{2}}+\frac{D_o}{2}\right)&\leq\frac{5}{2}e^{-\eta}\,.
\end{align}

The problem with the above inequalities is that they involve the unknown normalization $\Stau $ which we would like to replace with an empirical quantity. The following lemma presents such an empirical inequality.
\begin{lem}\label{lem:DKW}
	For every $\eta>0$, we have
	\begin{align*}
	P\left(\sn\Htauh\|\hatG-G\|_{\infty}\geq\sqrt{2\eta}+D_o\right)\leq \frac{7}{2}e^{-\eta}\,.
	\end{align*}
\end{lem}
See proof in Appendix~\ref{subsec:proof_DKW}.

Using this empirical version of DKW-KM, we are now ready to present an empirical Hoeffeding-type inequality for censored observations.
\begin{thm}\label{thm:hoeffding_empirical}
	Assume (A\ref{as:independentCT})--(A\ref{as:no_simultanious_failure}) and that $\|f\|_{\infty}\leq  M$. Then for any $n\geq 1$, and $\eta>0$, we have
	\begin{align*}
	P\left(\frac{\sn\Htauh{\hat G_{\tau}^2}}{M}\left|\hat{\mu}_n(f)-\mu(f) \right|\geq 4\sqrt{2\eta}+3D_o\right)\leq \frac{11}{2}e^{-\eta}\,.
	\end{align*}
\end{thm}
See proof in Appendix~\ref{subsec:proof_hoeffding_empirical}.

\subsection{Bound on Classes of Functions}\label{sec:uniform}
So far we bounded the difference between the estimator $\hat{\mu}_n(f)$ and its expectation $\mu(f) $ for some function $f$. In this section we would like to bound this difference over a class of functions. More specifically, let $\F$ be a class of bounded functions such that~\eqref{eq:f} holds for every $f\in\F$. We are interested in bounding
$ \sup_{f\in\F} |\hat{\mu}_n(f)-\mu(f)|$. When $\F$ is finite, one can use the union bound:
\begin{align}\label{eq:union_bound}
\begin{split}
&P\left(\sup_{f\in\F}\frac{\sn \Htauh{\hat G_{\tau}^2}}{M}\left|\hat{\mu}_n(f)-\mu(f)\right|\geq 4\sqrt{2\eta}+3D_o\right)\\
&\quad\leq \sum_{f\in\mathcal{F}}P\left(\frac{\sn \Htauh{\hat G_{\tau}^2}}{M}\left|\hat{\mu}_n(f)-\mu(f)\right|\geq 4\sqrt{2\eta}+3D_o\right)
\leq|\F|\frac{11}{2}e^{-\eta}\,,
\end{split}
\end{align}
where $|\F|$ is the cardinality of $\F$. 

When $\F$ is not finite, we  need to bound the complexity of $\F$. Define an $\eps$-bracket in $L_2(P)$ as a pair of functions $l, u \in L_2(P)$ such that $P(l(X) \leq u(X))= 1$ and $P (\{l(X)-u(X)\}^2)^{1/2}\allowdisplaybreaks\leq \eps$. Define the bracketing number $N_{[]}(\eps,\F,L_2(P))$ to be the minimum number of $\eps$-brackets in $L_2(P)$ needed to ensure that every
$f \in \F$ lies in at least one bracket \citep[Chapter~2.2]{Kosorok08}. 
\begin{thm}\label{lem:class_of_funs}
	Assume (A\ref{as:independentCT})--(A\ref{as:no_simultanious_failure}) and that $\|f\|_{\infty}\leq  M$. Assume also that $
	\log\mathcal{N}_{[]}(\eps,\F,L_2(P))<\gamma/\eps$ for some $\gamma>0$. Then, for any $n\geq 1$ and $\eta>0$, we have
	\begin{align*}
	P\left( \sup_{f\in\mathcal{F}}\frac{  \sn\Htau^2}{M}\left|\hat{\mu}_n(f)-\mu(f) \right|\geq 3\sqrt{\frac{\eta}{2}} +2D_o+2\right)\leq 5e^{-\eta}+e^{-\Htau/(3n)}\,,
	\end{align*}
	and
	\begin{align*}
	P\left(\frac{\sn\Htauh{\hat G_{\tau}^2}}{M}\sup_{f\in\F}\left|\hat{\mu}_n(f)-\mu(f) \right|\geq 4\sqrt{2\eta}+4D_o\right)
	\leq 6 e^{-\eta}\,.
	\end{align*}  
\end{thm}
See proof in Appendix~\ref{subsec:proof_class}.

\subsection{Bernstein-Type Inequality}\label{sec:Bernstein}

Using the Hoeffding inequalities we derived in the previous sections, we are now able to prove the following  Bernstein-type inequality for right-censored observations.
\begin{thm}\label{thm:Bernstein}
	Assume (A\ref{as:independentCT})--(A\ref{as:no_simultanious_failure}), and that $\|f\|_{\infty}\leq  M$. Then for any $n\geq 1$ and $\eta>0$, we have
	\begin{align*}
	P\left( |\hat{\mu}_n(f)-\mu(f) |\geq  \sqrt{\frac{2\sigma_f^2\eta}{n}}+\frac{M}{n\Htau}\left(2\eta+\frac{\Stau+\Stauh}{\Htau\Htauh}\left(3\sqrt{\frac{\eta}{2}}+2D_o\right)\sqrt{2\eta}\right)
	\right)\leq\frac{23}{2}e^{-\eta},
	\end{align*}
	where $\sigma_f^2$ is defined in~\eqref{eq:sigma_f}.
\end{thm}
See proof in Appendix~\ref{subsec:Bernstein}. 

Replacing the empirical term $\Htauh$ in Theorem~\ref{thm:Bernstein} can be done using the multiplicative Chernoff bound \citep{Hagerup1990}.
\begin{cor}\label{cor:Bernstein}
	Assume (A\ref{as:independentCT})--(A\ref{as:no_simultanious_failure}), and that $\|f\|_{\infty}\leq  M$. Then for any $n\geq 4$ and $\eta>0$, we have
	\begin{align*}
	P\left( |\hat{\mu}_n(f)-\mu(f) |\geq  \sqrt{\frac{2\sigma_f^2\eta}{n}}+\frac{M}{n\Htau}\left(2\eta+\frac{2}{H_{\tau}^2}\left(3\sqrt{\frac{\eta}{2}}+2D_o+3\right)\sqrt{2\eta}\right)
	\right)\leq\frac{23}{2}e^{-\eta}+ 2e^{-\Htau/(3n)},
	\end{align*}
	where $\sigma_f^2$ is defined in~\eqref{eq:sigma_f}.
\end{cor}
See proof in Appendix~\ref{subsec:proof_cor_Bernstein}.

\subsection{Discussion}\label{subsec:discussion} We compare the concentration inequalities presented in this section to their standard counterparts that appear in~\eqref{eq:Hoeffding} and~\eqref{eq:Bernstein}. All the inequalities present an exponential bound on the probability that the  difference between the expectation and the estimator of the expectation is larger than some term of interest. For the censored-data inequalities, the estimator $\hat{\mu}_n$ is defined in~\eqref{eq:estimator}, while for the standard inequalities the estimator is the sample mean $\bar{\mu}_n$. For the Hoeffding inequalities, the term of interest depends on the magnitude $M$ of the function $f$ scaled by $n^{-1/2}$. For the censored-data Hoeffding inequalities, the term of interest also includes the universal constants $D_o$, and either constants that depend on the distribution such as $\Htau$ and $\Gtau$, or their empirical counterparts. The term of interest of the Bernstein inequalities is a sum of two expressions. The first depends on the (possibly asymptotic) variance of the estimator. The second expression depends on the magnitude $M$ of the function $f$ scaled by $n^{-1}$. In the censored-data Bernstein inequality, the second expression also depends on the universal constant $D_o$, and on the distribution-dependent constant $\Htau$. 

The bounds given here are not sharp. The main difficulty we face is that the term $\Ghat_n(T-)$ that appears in the denominator of the estimator $\hat\mu_n(f)$ is not bounded from below.  While improvements to the constants can be obtained, we preferred to keep the results and the proofs relatively simple. Moreover, it appears that constants that depend on the distribution cannot be eliminated. This can be observed, for example, in the  Dvoretzky-Kiefer-Wolfowitz-type
inequality for Kaplan-Meier estimator. Note also that while the $n^{-1/2}$ term in the censored Bernstein inequality is optimal up to multiplication by a constant, it looks like the expression $1/\Htau$ that appears in the $n^{-1}$ term  is necessary. 


\section{Application to Empirical Risk Minimization}\label{sec:functions}
In the following section we show how to apply the Hoeffding inequalities which appear in Section~\ref{sec:single} to empirical risk minimization (ERM). 

Assume that the data consist of $n$ i.i.d.\ random $4$-tuples $\{(U_1,\delta_1,Z_1,Y_1),\allowbreak \ldots,(U_n,\delta_n,Z_n,Y_n)\}$ where $Y_i$, $i=1,\ldots,n$, are random responses that get their values in $\R$. Let $\F$ be a compact class of integrable functions, such that~\eqref{eq:f} holds for every $f\in\F$. Let $L:\R\times\R\mapsto [0,\infty)$ be a locally Lipschitz continuous loss such that
\begin{align}\label{eq:Lipchitz}
\begin{split}
&L( y, f(t,z)) \leq B\,,\qquad\quad (y,t,z)\in\R\times[0,\tau]\times \Z, f\in\F\,,\\
& \sup_{y}|L(y,s)-L(y,s')|\leq L_M|s-s'|\,, \qquad s,s'\in[-M,M]
\end{split}
\end{align}  
for some constants $B$ and $L_M$. Define the risk with respect to the loss $L$ as $R_L(f)\equiv P[L(Y,f(T,Z))]$. 

When no censoring is introduced, define the empirical risk minimizer 
\begin{align*}
f_n\equiv\argmin_{f\in\F} \ep_n L(Y,f(T,Z))\,,
\end{align*}
which exists by compactness of $\F$. Define the $\eps$-covering number of $\F$ with respect to the sup norm, denoted by $\mathcal{N}(\F,\|\cdot \|_{\infty},\eps)$, as the size of the smallest covering of $\F$ with sup norm balls of radius $\eps>0$. Assume that for all $\eps>0$
\begin{align}\label{eq:covering_num}
\log\mathcal{N}(\F,\|\cdot \|_{\infty},\eps)\leq a\eps^{-p}
\end{align}
for some constants $a>1$ and $p>0$. Then it can be shown that $P\big(R_L(f_n)>\inf_{f\in\F} R_L(f)+\eps\big)\rightarrow 0$ for every $\eps>0$ \citep[][Proposition~6.22]{SVR}.

For censored data we define the function $L\circ f$ as $ (y,t,z)\mapsto L(y,f(t,z))$ for every $f\in\F$ and $(y,t,z)\in\R\times[0,\tau]\times\Z$. Define the empirical risk minimizer 
$f_{n}^C$ for right-censored data as
$$f_{n}^C\equiv \argmin_{f\in\F} \hat{\mu}_n(L\circ f)\,.$$
We have the following oracle inequality.
\begin{lem}\label{lem:erm}
	Assume (A\ref{as:independentCT})--(A\ref{as:no_simultanious_failure}) and the loss function $L$ satisfies~\eqref{eq:Lipchitz}. Then, for any $n\geq 1$, and $\eta>0$, with probability of not less than $1-\frac{11}{2}e^\eta$,
	\begin{align*}
	\frac{\Htauh{\hat G_{\tau}^2}}{B}\left|R_L(f_n^C)-\inf_{f\in\mathcal{F}}R_L(f)\right|
	\leq\frac{8\sqrt{2\eta+2\log\mathcal{N}(\F,\|\cdot \|_{\infty},\eps)}+6D_o}{\sn}+\frac{4\eps L_M}{B}\,.
	\end{align*}  
\end{lem}
See proof in Appendix~\ref{subsec:proof_erm}. As an immediate corollary we obtain that the minimizer of the empirical risk $f_n^C$ is consistent.
\begin{cor}
	Assume (A\ref{as:independentCT})--(A\ref{as:no_simultanious_failure}) and that~\eqref{eq:covering_num} holds, then
	$$	P\left( R_L(f_n^C)>\inf_{f\in\F}R_L(f)+\eps\right)\rightarrow 0\,.
	$$
\end{cor}

\section{Concluding Remarks}\label{sec:summary}
We presented Hoeffding-type and Bernstein-type inequalities for right-censored data. While the inequalities are not sharp, they provide a simple tool for developing and analyzing machine learning algorithms for right-censored data. 

Future research questions include sharpening the inequalities obtained in this work (see discussion in Section~\ref{subsec:discussion}). Another research direction is relaxation  of the assumption that the failure and censoring are independent. This can be done, for example, by assuming a model on the censoring distribution, or by using generalized Kaplan-Meier estimators~\citep{Dabrowska87,Dabrowska89}. A third research question is how to derive other concentration inequalities such as Talagrand's inequality.


\appendix
\section{Proofs}\label{sec:proofs}

\subsection{Proof of Theorem~\ref{thm:asymptotic}}\label{subsec:asymptotic}
Write
\begin{align*}
\hat{\mu}_n(f) - \mu(f) = &(\ep_n-P)\left(\frac{\delta f(T,Z)}{    \hatG(T-)}-\frac{\delta f(T,Z)}{    G(T-)}\right)
+(\ep_n-P)\frac{\delta f(T,Z)}{    G(T-)}
+P\left(\frac{\delta f(T,Z)}{    \hatG(T-)}-\frac{\delta f(T,Z)}{    G(T-)}\right)\\
=&A_n+B_n+C_n\,.
\end{align*}
Note that
\begin{align}\label{eq:AofConsistency}
|A_n|\leq \sup_{t\in[0,\tau)}|G(t)-\hatG(t) | \,\left|(\ep_n-P)\frac{\delta f(T,Z)}{   G(T-) \hatG(T-)}\right|\,,
\end{align}
which converges to zero by the consistency of the Kaplan-Meier estimator and Assumption~(A\ref{as:positiveRisk}). A similar argument shows that $C_n$ converges to zero. Finally, $B_n$ converges to zero by the law of large numbers, which complete the proof of consistency.

For the normality part, we follow \citet{Bang2000}. Let the filtration $\mathcal{G}(t)$ be the set of $\sigma$-algebras generated by 
\begin{align*}
\Big\{ \indi{C_i\leq t}, \indi{T_i\leq t},Z_i;\, t\in[0,\tau]\,,i\in 1,\ldots,n\Big\}\,.
\end{align*}
Consider the martingale $M_i^C(t)= N_i^C(t)-\int_0^t Y_i(s)d\Lambda^C(s)$, where $\Lambda^C$ is the cumulative hazard function for the censored distribution, 
$N_i^C(t)=\indi{U_i\leq t,\delta=0}$ and
$Y_i(t)=\indi{U\geq t}$. Note that $N_i^C(t)$ is the counting process for the censoring, and not for the failure events, and by Assumption~(A\ref{as:no_simultanious_failure}), $Y_i(t)$ is the at-risk process for observing a censoring time.  Define $M^C(t)=\sum M_i^C(t)$, $N^C(t)=\sum N_i^C(t)$, and $Y(t)=\sum Y_i(t)$. Using identity~\eqref{eq:3.10.d} below,	the martingale integral representation \citep[][p.~37]{gill1980censoring}
\begin{align*}
\frac{\hatG(t)-G(t)}{G(t)}=-\int_{0}^{t}  \frac{\hatG(u-)dM^C(u)}{G(u)Y(u)}\,,
\end{align*}
and the relationship 
\begin{align}\label{eq:identity}
n^{-1}Y(t)=\hatG(t-)\hat{S}(t-)\,,
\end{align}
where $\hat{S}$ is the Kaplan-Meier estimator of $S$, we can write
\begin{align}\label{eq:diff_mun_mu}
\begin{split}
&  \sn (\hat{\mu}_n(f)-\mu(f) )
\\
&\quad =n^{-1/2}\sum_{i=1}^{n}\frac{\delta_i f(T_i,Z_i)}{G(T_i)}- 
n^{-1/2}\sum_{i=1}^{n}\frac{\delta_i f(T_i,Z_i)}{\hatG(T_i)}\frac{\hatG(T_i)-G(T_i)}{G(T_i)}
-n^{1/2}\mu(f) 
\\
&\quad =n^{-1/2}\sum_{i=1}^{n}\left(f(T_i,Z_i)-\mu(f) \right)\\
&\qquad -n^{-1/2}\sum_{i=1}^n\int_0^\tau f(s,Z_i)\frac{dM_i(s)}{G(s)}+n^{-1/2} \frac{1}{n}\sum_{i=1}^{n}\frac{\delta_i f(T_i,Z_i)}{\hatG(T_i)}\int_0^\tau\frac{\indi{T_i\geq s}}{\hat{S}(s-)G(s)}dM^C(s)
\\
&\quad =n^{-1/2}\sum_{i=1}^n \left(f(T_i,Z_i)-\mu(f) \right)
\\
&\qquad-n^{-1/2}\sum_{i=1}^n\int_0^\tau \left(f(s,Z_i)-\frac{P[f(T,Z)\indi{T\geq s}]}{S(s)}\right)\frac{dM_i(s)}{G(s)}\\
&\qquad+n^{-1/2} \int_0^\tau\left[\frac{1}{n}\sum_{i=1}^{n}\left(\frac{\delta_i f(T_i,Z_i)}{\hatG(T_i)}\frac{\indi{T_i\geq s}}{\hat{S}(s-)}\right)- \frac{P[f(T,Z)\indi{T\geq s}]}{S(s-)}\right]\frac{dM^C(s)}{G(s)}
\\
&	
\quad \equiv A_n-B_n+C_n\,.
\end{split}
\end{align}
Note that $C_n$ in the above equation is $o_P(1)$. Since $A_n$ in~\eqref{eq:diff_mun_mu} is $\mathcal{G}(0)$-predictable and $B_n$ is a martingale, $A_n$ and $B_n$ are uncorrelated sums of i.i.d.\ random variables. The variance of $A_n$ is $\mathrm{Var}(f(T,Z))$. By Theorem~2.6.2 of~\citet{FH} and the central limit theorem, 
\begin{align*}
&  \mathrm{Var}\left(\int_0^\tau \left(f(s,Z)-\frac{P[f(T,Z) \indi{T\geq s}]}{S(s-)}\right)\frac{dM(s)}{G(s)}\right)\\
&\quad=P\left[\int_0^\tau\left(f(s,Z)-\frac{P[f(T,Z) \indi{T\geq s}]}{S(s-)}\right)^2\indi{T\geq s}\indi{C\geq s}\frac{(1-\Delta\Lambda(s)))d\Lambda^C(s)}{G^2(s)}\right]\,,
\end{align*}
where for a function $F$, $\Delta F (s)=F(s)-F (s-)$. Taking the conditional expectation with respect to $C$ and noting that (see Eq.~\ref{eq:int_rep})
\begin{align*}
G(s)=G(s-)(1-\Delta\Lambda(s))
\end{align*} completes the normality part.

We now show that   $
\hat{\mu}_n(f) = \int_0^\tau f(t) d\hat F_n(t)$ when $f$ is a function only of $T$. First, note that the Kaplan-Meier estimator $\hat F_n$ has jumps only at time-points in which there are uncensored observations~\citep{Efron1967}. Let $t^*_j$, $j=1,\ldots,m$ be the ordered times in which there are uncensored observations and let $n^*_j$ be the number of failures that occurred at time $t^*_j$. Then
\begin{align*}
\int_0^\tau f(t) d\hat F_n(t)= \sum_{j=1}^m f(t^*_j) \Delta \hat F_n(t^*_j)\,.
\end{align*}
By the product-limit representation of the Kaplan-Meier estimator \citep[see][Eq.~6]{Satten01}
\begin{align*}
\Delta \hat F_n(t^*_j)= \hat S_n(t^*_j-)\frac{n^*_j}{Y(t^*_j)}\,,
\end{align*}
which by~\eqref{eq:identity} can be written as
\begin{align*}
\Delta \hat F_n(t^*_j)=\frac{1}{n}\frac{n^*_j}{\hatG(t^*_j)}\,.
\end{align*}
Hence
\begin{align}\label{eq:km-rep}
\int_0^\tau f(t) d\hat F_n(t)= \frac{1}{n}\sum_{j=1}^m f(t^*_j)\frac{n^*_j}{\hatG(t^*_j)}=\frac{1}{n}\sum_{i=1}^n \frac{\delta_if(T_i)}{\hatG(T_i)}
=\hat{\mu}_n(f)
\end{align}
by Assumption~(A\ref{as:no_simultanious_failure}).

To see that $\hat{\mu}_n(f)$ is an efficient estimator for $\mu(f)$, note that by~\eqref{eq:km-rep}, $\hat{\mu}_n(f)$ is a functional of the Kaplan-Meier estimator. Since the Kaplan-Meier estimator is efficient~\citep[][Chapter~4.3]{Kosorok08}, and the functional $\phi:F\mapsto\int f(t)dF(t)$ is Hadamard differentiable, \citep[][Lemma~12.3]{Kosorok08}, the result follows from Theorem~18.6 of \citet{Kosorok08}. 

\subsection{Proof of an Identity}
We prove identity 3.10.d in~\citet[][p.~313]{robins1992recovery} for general distributions. 
\begin{lem}
	Let $G$, the cumulative distribution function of $ C$, be right-continuous and differentiable except at
	points in a countably infinite set. Then
	\begin{align}\label{eq:3.10.d}
	\frac{\delta_i}{G(T_i)}=1-\int \frac{dM_i^C(t)}{G(t)}\,.
	\end{align}
\end{lem}
\begin{proof}
	First, we show that 
	\begin{align}\label{eq:3.10.a}
	\frac{1}{G(t)}= \frac{1}{G(u)}+\int_u^t \frac{d\Lambda(s)}{G(s)}\,.
	\end{align}
	We first assume that $G$ that is differentiable on the segment $[u,\tau]$  except at a finite set of points $u=t_0< t_1<\ldots<t_m=\tau$. Denote by $g$ the derivative of $1-G$ at all points in which it exists. Let $\lambda(s)=g(s)/G(s-)$. Then, by Appendix~A.1 of \citet{FH}, we can write~\eqref{eq:3.10.a} as
	\begin{align*}
	\frac{1}{G(t)}= \frac{1}{G(u)}+\int_u^t \frac{\lambda(s)ds}{G(s)}+\sum_{i=1}^m \frac{\Delta\Lambda(t_i)}{G(t_i)}=\sum_{i=1}^{ m}\left(\int_{t_{i-1}}^{t_i}\frac{\lambda(s)ds}{G(s)}+\frac{\Delta\Lambda(t_i)}{G(t_i)}\right)\,.
	\end{align*}
		Using integration, it is easy to show that \citep[ Eq.~3.10.a]{robins1992recovery}
	\begin{align*}
	\int_{t_{i-1}}^{t_i}\frac{\lambda(s)ds}{G(s)}=\frac{1}{G(t_i-)}-\frac{1}{G(t_{i-1})}\,.
	\end{align*}
	Using the product integral representation \citep[Chapter~12.2.3]{Kosorok08}, we can write
	\begin{align*}
	G(t_i)= e^{-\int_0^{t_i}\lambda(s)ds }\prod_{t_j:j\leq i}\left\{1-\Delta\Lambda(t_j)\right\}\,.
	\end{align*}
	Hence,
	\begin{align}\label{eq:int_rep}
	\begin{split}
	\frac{\Delta\Lambda(t_i)}{G(t_i)}&=\frac{\Delta\Lambda(t_i)}{e^{-\int_0^{t_i}\lambda(s)ds }\prod_{t_j:j\leq i}\left(1-\Delta\Lambda(t_j)\right)}\\
	&=\frac{1-(1-\Delta\Lambda(t_i))}{\left\{e^{-\int_0^{t_i}\lambda(s)ds }\prod_{t_j:j< i}\left(1-\Delta\Lambda(t_j)\right)\right\}\left(1-\Delta\Lambda(t_i)\right)}=\frac{1}{G(t_i)}-\frac{1}{G(t_i-)}\,.
	\end{split}
	\end{align}
	Thus, we conclude that
	\begin{align*}
	\int_{t_{i-1}}^{t_i}\frac{\lambda(s)ds}{G(s)}+\frac{\Delta\Lambda(t_i)}{G(t_i)}= \frac{1}{G(t_i)}-\frac{1}{G(t_{i-1})}\,,
	\end{align*}
	which proves~\eqref{eq:3.10.a} when $G$ is differentiable except at a finite set of points. The result also follows when the set at which $G$ is not differentiable is countable by bounding $G$ from below and from above by functions that have finite sets of non-differentiability, and taking the limit.
	
	Write
	\begin{align*}
	\frac{\delta_i}{G(U_i)}+\frac{1-\delta_i}{G(U_i)}=\frac{1}{G(U_i)}= \frac{1}{G(0)}+\int_0^{U_i} \frac{d\Lambda(s)}{G(s)}\,.
	\end{align*}
	Hence,
	\begin{align*}
	\frac{\delta_i}{G(U_i)}=1 - \left(\frac{1-\delta_i}{G(U_i)}-\int\frac{\indi{U_i\geq s}d\Lambda(s)}{G(s)}\right)=1-\int\frac{dM_i^C(s)}{G(s)}\,.
	\end{align*}
	The result follows since $\delta_i/G(U_i)=\delta_i/G(T_i)$.
\end{proof}

\subsection{Proof of Theorem~\ref{thm:hoeffding}}\label{subsec:proof_hoeffding}
Define
\begin{align*}
\Omega_n=\left\{\Stau \sup_{0< t\leq \tau}| \hatG(t-)- G(t-)|\leq \frac{\sqrt{\eta/2}+D_o/2}{\sn}\right\}\,,
\end{align*}
and note that by~\eqref{eq:KM_exp_bound},
\begin{align}\label{eq:prob_omegan_single}
P(\Omega_n) >1-
\frac{5}{2}e^{-\eta}\,.
\end{align}

Write
\begin{align}\label{eq:bound_by_AnBn}
\left|\hat{\mu}_n(f)-\mu(f) \right|\leq & \left|\ep_n \frac{\delta f(T,Z)}{    \hatG(T-)}-\ep_n \frac{\delta f(T,Z)}{    G(T-)}\right|
+\left|\ep_n \frac{\delta f(T,Z)}{    G(T-)}-P \frac{\delta f(T,Z)}{    G(T-)}\right|\equiv A_n(f)+B_n(f)\,,
\end{align}
where we use the fact $\mu(f) =P[\delta f(T,Z)/    G(T-)]$, which follows from Assumption~(A\ref{as:independentCT}) by conditional expectations. Then
\begin{align}\label{eq:main_single}
\begin{split}
&P\left(\frac{\sn \Htau\Gtauh}{M}\left|\hat{\mu}_n(f)-\mu(f) \right|\geq 3\sqrt{\frac{\eta}{2}} +\frac{D_o}{2}\right)
\\
&
\leq
P\left(\left.\frac{\sn \Htau\Gtauh}{M}A_n(f)\geq \sqrt{\frac{\eta}{2}} +\frac{D_o}{2}\right|\Omega_n\right)P(\Omega_n)
+P\left(\left.\frac{\sn \Htau\Gtauh}{M}B_n(f)\geq\Stau\sqrt{2\eta}\right|\Omega_n\right)P(\Omega_n)+P(\Omega_n^C)\,.
\end{split}\end{align}

We start by bounding the first expression. Note that
\begin{align}\label{eq:bound_A_n}
A_n(f)&\leq \ep_n  \left|\frac{\delta f(T,Z)}{G(T-)\hatG(T-)}\right|
|G(T-)-\hatG(T-)|
\leq \frac{M}{\Gtauh\Gtau}\|G-\hatG\|_{\infty}\,.
\end{align}
Hence,
\begin{align}\label{eq:bound_A_n2}
\begin{split}
&P\left(\left.\frac{\sn \Htau\Gtauh}{M}A_n(f)\geq
\sqrt{\frac{\eta}{2}}+\frac{D_o}{2}\right|\Omega_n\right)P(\Omega_n)
\leq   P\left(\left.\sn \Stau \|G-\hatG\|_{\infty}\geq\sqrt{\frac{\eta}{2}} +\frac{D_o}{2} \right|\Omega_n\right)P(\Omega_n)
=0\,,
\end{split}
\end{align}
where the inequality follows from~\eqref{eq:bound_A_n}, and the equality follows from the definition of $\Omega_n$.

Recall that
$B_n(f)\equiv
|(\ep_n -P)\delta f(T,Z)/ G(T-)|$. Hence
\begin{align}\label{eq:bound_B_n}
\begin{split}
&  P\left(\left.\frac{\sn \Gtauh\Htau}{M}B_n(f)\geq\Stau\sqrt{2\eta}\right|\Omega_n\right)P(\Omega_n)
\leq P\left(\frac{\sn \Stau\Gtau}{M}\left|(\ep_n -P)\frac{\delta f(T,Z)}{ G(T-)}\right|\geq\Stau\sqrt{2\eta}\right)\\
&
\leq P\left(\left|(\ep_n -P)\frac{\delta f(T,Z)}{ G(T-)}\right|\geq\frac{M}{\Gtau}\sqrt{\frac{2\eta}{n}}\right)\\
&\leq 2e^{-\eta}\,,
\end{split}
\end{align}
where the last inequality follows from the standard Hoeffding inequality~\eqref{eq:Hoeffding}, and where we used the fact that $|\delta f(T,Z)/ G(T)|<M/\Gtau$.

The result follows by substituting~\eqref{eq:prob_omegan_single},~\eqref{eq:bound_A_n2}, and~\eqref{eq:bound_B_n}, in~\eqref{eq:main_single}.

\subsection{Proof of Corollary~\ref{thm:hoeffding_cor}}\label{subsec:proof_hoeffding_cor}
Define $\Psi_n=\{\Htau \leq (1+\frac{1}{\sqrt{n}})\Htauh\}$ and note that by the multiplicative Chernoff bound \citep[Eq.~6]{Hagerup1990}
\begin{align}\label{eq:apply_multi_ch}
P(\Psi_n)\geq 1- e^{-\Htau/(3n)}\,.
\end{align} 

Write
\begin{align}\label{eq:main_new}
\begin{split}
&P\left(\frac{\sn \Htau^2}{M}\left|\hat{\mu}_n(f)-\mu(f) \right|\geq 3\sqrt{\frac{\eta}{2}} +\frac{D_o}{2}+2\right)
\\
&
\leq P\left(\left.\frac{\sn \Htau^2}{M}\left|\hat{\mu}_n(f)-\mu(f) \right|\geq 3\sqrt{\frac{\eta}{2}} +\frac{D_o}{2}+2\right|\Psi_n\right)P(\Psi_n)+P(\Psi_n^C)
\\
&
\leq P\left(\frac{\sn \Htau\left(1+\frac{1}{\sn}\right)\Htauh)}{M}\left|\hat{\mu}_n(f)-\mu(f) \right|\geq 3\sqrt{\frac{\eta}{2}} +\frac{D_o}{2}+2\right)+P(\Psi_n^C)
\\
&
\leq P\left(\frac{\sn \Htau\Gtauh}{M}\left|\hat{\mu}_n(f)-\mu(f) \right|\geq 3\sqrt{\frac{\eta}{2}} +\frac{D_o}{2}+2- \frac{\Htau\Htauh}{M}\left|\hat{\mu}_n(f)-\mu(f)\right| \right)+P(\Psi_n^C)
\\
&
\leq P\left(\frac{\sn \Htau\Gtauh}{M}\left|\hat{\mu}_n(f)-\mu(f) \right|\geq 3\sqrt{\frac{\eta}{2}} +\frac{D_o}{2}\right)+P(\Psi_n^C)
\end{split}\end{align}
where we used the fact that $\Gtauh\geq \Htauh$ and that $\frac{\Htau\Htauh}{M}\left|\hat{\mu}_n(f)-\mu(f)\right|\leq 2$. By Theorem~\ref{thm:hoeffding} and~\eqref{eq:apply_multi_ch}, the last line of~\eqref{eq:main_new} is bounded by $(9/2) e^{-\eta}+e^{-\Htau/3 n}$, which concludes the proof.

\subsection{Proof of~\eqref{eq:KM_exp_bound}}\label{subsec:proof_of_eq}
	For $\eta>0$, write
	\begin{align}\label{eq:epsilon}
	\begin{split}
	\eps=\sqrt{\frac{\eta}{2}+\left(\frac{D_o}{4}\right)^2}+\frac{D_o}{4}
	\Leftrightarrow  &\eps  -\frac{D_o}{4}=\sqrt{\frac{\eta}{2}+\left(\frac{D_o}{4}\right)^2}\\
	\Leftrightarrow  &\eps^2  -2\eps  \frac{D_o}{4} + \left(\frac{D_o}{4}\right)^2=\frac{\eta}{2}+\left(\frac{D_o}{4}\right)^2\\
	\Leftrightarrow& 2\eps^2 -\eps   D_o= \eta
	\,.
	\end{split}
	\end{align}
	Using the fact that $\sqrt{x+y}\leq\sqrt{x}+\sqrt{y}$ we obtain that
	\begin{align*}
	\eps=\sqrt{\frac{\eta}{2}+\left(\frac{D_o}{4}\right)^2}+\frac{D_o}{4}\leq \sqrt{\frac{\eta}{2}}+\frac{D_o}{4}+\frac{D_o}{4} = \sqrt{\frac{\eta}{2}}+\frac{D_o}{2}\,,
	\end{align*}
	Consequently, for $\eta=2 \eps^2-  D_o\eps$,	
	\begin{align*}
	P\left(\sn\Stau \|\hatG-G\|_{\infty}\geq\sqrt{\frac{\eta}{2}}+\frac{D_o}{2}\right)\leq P\left(\sn\Stau \|\hatG-G\|_{\infty}\geq\eps\right)&\leq\frac{5}{2}e^{-\{2 \eps^2-  D_o\eps\}}\leq\frac{5}{2}e^{-\eta}\,.
	\end{align*}
\subsection{Proof of Lemma~\ref{lem:DKW}}\label{subsec:proof_DKW}
	Define $\Omega_n=\{\Htau+ \eps/(2\sn)\geq\Htauh \} $, and note that by the standard Hoeffeding inequality $P(\Omega_n)>1-e^{-\eps^2/2}$.
	Write 
	\begin{align*}
	P\left(\sqrt{n}\Htauh\|\hatG-G\|_{\infty}\geq\eps\right)
	&\leq P\left(\left.\sqrt{n}\Htauh\|\hatG-G\|_{\infty}\geq\eps\right|\Omega_n\right)P(\Omega_n)+e^{-\eps^2/2}
	\\
	&\leq P\left(\left.\sqrt{n}\left(\Htau+\eps/(2\sn)\right)\|\hatG-G\|_{\infty}\geq\eps\right|\Omega_n\right)P(\Omega_n)+e^{-\eps^2/2}\\
	&\leq P\left(\sqrt{n}\Htau\|\hatG-G\|_{\infty}\geq\eps\bigr(1-\frac{1}{2}\|\hatG-G\|_{\infty}\bigr)\right)+e^{-\eps^2/2}\\
	&\leq P\left(
	\sn  \Stau  \|\hatG-G\|_{\infty}\geq\frac{\eps}{2}\right)+e^{-\eps^2/2}\\
	&\leq e^{-\eps^2/2}\left(1+\frac{5}{2}e^{D_o\eps/2}\right)<\frac{7}{2}e^{-\eps^2/2+D_o\eps/2}\,,
	\end{align*}
	where we used the fact that $\|\hatG-G\|_{\infty}<1$.
	Similarly to~\eqref{eq:epsilon}, for $\eta>0$,
	write
	\begin{align}\label{eq:eps2eta}
	\frac{\eps}{2}=\sqrt{\frac{\eta}{2}+\left(\frac{D_o}{4}\right)^2}+\frac{D_o}{4} \Rightarrow \eta=\frac{\eps^2}{2}- \frac{D_o\eps}{2}\,,
	\end{align}
and note that $\eps\leq \sqrt{2\eta}+D_o$. Thus,
	\begin{align*}
	P\left(\sn\Htauh\|\hatG-G\|_{\infty}\geq\sqrt{2\eta}+D_o\right)
	&\leq  P\left(\sn\Htauh\|\hatG-G\|_{\infty}\geq\eps\right)
	\leq \frac{7}{2}e^{-\eps^2/2+ D_o\eps/2} =\frac{7}{2}e^{-\eta}
	\,,
	\end{align*}
	which completes the proof.

\subsection{Proof of Theorem~\ref{thm:hoeffding_empirical}}\label{subsec:proof_hoeffding_empirical}
	Define
	\begin{align}\label{eq:Delta_n}
	\Delta_n=\left\{\Htauh\sup_{0< t\leq \tau}| \hatG(t-)- G(t-)|\leq \frac{\sqrt{2\eta}+D_o}{\sn}\right\}\,,
	\end{align}
	and note that by Lemma~\ref{lem:DKW},
	\begin{align}\label{eq:prob_omegan}
	P(\Delta_n) \geq 1-
	\frac{7}{2}e^{-\eta}\,.
	\end{align}
	Let
	\begin{align*}
	\Gamma_n=\left\{\Htauh\Gtauh<\Htauh\Gtau+ \frac{\sqrt{2\eta}+D_o}{\sn}\right\}
	\end{align*}
	and note $\Delta_n\subset\Gamma_n$. Write
	\begin{align*}
	&P\left(\frac{\sn\Htauh{\hat G_{\tau}^2}}{M}\left|\hat{\mu}_n(f)-\mu(f) \right|\geq 4\sqrt{2\eta}+3D_o\right)\\
&   \leq
	P\left(\frac{\sn\Htauh{\hat G_{\tau}^2}}{M}\left|\hat{\mu}_n(f)-\mu(f) \right|\geq \right.
	\\
	&\left.\qquad\qquad\qquad\sqrt{2\eta} +D_o+ \frac{\Gtauh\left|\hat{\mu}_n(f)-\mu(f) \right|}{2M}2\left(\sqrt{2\eta} +D_o\right)+\Gtauh\Htauh\sqrt{2\eta}\right)
	\\
	&
	\leq P\left(\frac{\sn\Gtauh}{M}\left(\Htauh\Gtauh-\frac{\sqrt{2\eta} +D_o}{\sn}\right)\left|\hat{\mu}_n(f)-\mu(f) \right|\geq \sqrt{2\eta} +D_o+\Gtauh\Htauh\sqrt{2\eta}\right)\,,
	\end{align*}
	where we used the fact that $\frac{\Gtauh\left|\hat{\mu}_n(f)-\mu(f) \right|}{2M}\leq 1$.
	Using the decomposition~\eqref{eq:bound_by_AnBn}, and conditioning on $\Delta_n$, we have 
	\begin{align}\label{eq:main_empirical}
	\begin{split}
	&
	P\left(\frac{\sn\Gtauh}{M}\left(\Htauh\Gtauh-\frac{\sqrt{2\eta} +D_o}{\sn}\right)\left|\hat{\mu}_n(f)-\mu(f) \right|\geq \sqrt{2\eta} +D_o+\Gtauh\Htauh\sqrt{2\eta}\right)
	\\
	&
	\leq
	P\left(\left.\frac{\sn\Gtauh}{M}\left(\Htauh\Gtauh-\frac{\sqrt{2\eta} +D_o}{\sn}\right)A_n(f)\geq \sqrt{2\eta} +D_o\right|\Delta_n\right)P(\Delta_n)\\
	&\quad+P\left(\left.\frac{\sn\Gtauh}{M}\left(\Htauh\Gtauh-\frac{\sqrt{2\eta} +D_o}{\sn}\right)B_n(f)\geq\Gtauh\Htauh\sqrt{2\eta}\right|\Delta_n\right)P(\Delta_n)\\
	&\quad+P(\Delta_n^C)\,.
	\end{split}
	\end{align}

	We start by bounding the first expression.
	\begin{align}\label{eq:bound_A_n_thm4}
	\begin{split}
	&P\left(\left.\frac{\sn\Gtauh A_n(f)}{M}\left(\Htauh\Gtauh-\frac{\sqrt{2\eta} +D_o}{\sn}\right)\geq
	\sqrt{2\eta}+D_o\right|\Delta_{n}\cap\Gamma_n\right)P(\Delta_n)
	\\
	&\leq   P\left(\left.\frac{\sn\Htauh\Gtauh\Gtau A_n(f)}{M}\geq\sqrt{2\eta} +D_o \right|\Delta_n\right)P(\Delta_n)\\
	&\leq   P\left(\left.\sn\Htauh\|G-\hatG\|_{\infty}\geq\sqrt{2\eta} +D_o \right|\Delta_n\right)P(\Delta_n)=0\,,
	\end{split}
	\end{align}
	where the first inequality follows from the definition of $\Gamma_n$, the second inequality follows from~\eqref{eq:bound_A_n}, and the equality follows from the definition of $\Delta_n$.

	Recall that
	$B_n(f)\equiv
	|(\ep_n -P)\delta f(T,Z)/ G(T-)|$. Hence
	\begin{align*}
	&  P\left(\left.\frac{\sn\Gtauh}{M}\left(\Htauh\Gtauh-\frac{\sqrt{2\eta} +D_o}{\sn}\right)B_n(f)\geq\Gtauh\Htauh\sqrt{2\eta}\right|\Delta_n\right)P(\Delta_n)\\
	&
	\leq P\left(\frac{\sn\Gtauh}{M}\Htauh\Gtau\left|\frac{(\ep_n -P)\delta f(T,Z)}{ G(T-)}\right|\geq\Gtauh\Htauh\sqrt{2\eta}\right)\\
	&
	\leq P\left(\left|(\ep_n -P)\frac{\delta f(T,Z)}{ G(T-)}\right|\geq\frac{M}{\Gtau}\sqrt{\frac{2\eta}{n}}\right)\leq 2e^{-\eta}\,,
	\end{align*}
	where the first inequality follows from the definition of $\Delta_n$, the second inequality follows from~\eqref{eq:bound_A_n}, and the last inequality follows from the standard Hoeffding inequality, where we use the fact that $|\delta f(T,Z)/ G(T-)|\leq M/\Gtau$.
	
	The result follows by substituting~\eqref{eq:prob_omegan} and the two bounds above in~\eqref{eq:main_empirical}.

\subsection{Proof of Theorem~\ref{lem:class_of_funs}}\label{subsec:proof_class}
	We prove the second assertion, the first follows using the same arguments. Using~\eqref{eq:main_empirical} one can show that
	\begin{align}\label{eq:prob_class_3terms}
	\begin{split}
	&P\left(\frac{\sn\Htauh{\hat G_{\tau}^2}}{M}\sup_{f\in\F}\left|\hat{\mu}_n(f)-\mu(f) \right|\geq 4\sqrt{2\eta}+4D_o\right)
	\\
	&
	\quad\leq
	P\left(\left.\frac{\sn\Htauh\Gtau\Gtauh}{M}\sup_{f\in\F}A_n(f)\geq \sqrt{2\eta} +D_o\right|\Delta_n\right)P(\Delta_n)\\
	&\qquad+P\left(\left.\frac{\sn\Htauh\Gtau\Gtauh}{M}\sup_{f\in\F}B_n(f)\geq\Gtauh\Htauh(\sqrt{2\eta}+D_0)\right|\Delta_n\right)P(\Delta_n)\\
	&\qquad+P(\Delta_n^C)\,,
	\end{split}
	\end{align}
	where $A_n$, and $B_n$ are defined in~\eqref{eq:bound_by_AnBn} and $\Delta_n$ is defined in~\eqref{eq:Delta_n}.
	For the first term in the RHS of~\eqref{eq:prob_class_3terms}, note that the bound of the first term in~\eqref{eq:bound_A_n_thm4} holds even when taking the supremum over the class $\F$. For the second term, by Corollary~2 of \citet{Bitouze99}, 
	\begin{align*}
	P\left(\frac{\sn\Gtau\sn}{2M} \sup_{f\in\F} \left|(\ep_n-P)\frac{ \delta f(T,Z)}{G(T-)} \right|\geq\eps\right)\leq \frac{5}{2} e^{-2\eps^2+D_o\eps}\,.
	\end{align*} 
	For $\eta>0$, define $\eps$ as in~\eqref{eq:epsilon}, and note
that $\eps<\sqrt{\eta/2}+D_o/2$.  Hence,
	\begin{align*}
	&P\left(\frac{\sn\Gtau}{M} \sup_{f\in\F} \left|(\ep_n-P)\frac{ \delta f(T,Z)}{G(T-)} \right|\geq\sqrt{2\eta}+D_o\right)
	= P\left(\frac{\sn\Gtau}{2M} \sup_{f\in\F} \left|(\ep_n-P)\frac{ \delta f(T,Z)}{G(T-)} \right|\geq\sqrt{\frac{\eta}{2}}+\frac{D_o}{2}\right) 
	\\
	&\leq P\left(\frac{\Gtau\sn}{2M} \sup_{f\in\F} \left|(\ep_n-P)\frac{ \delta f(T,Z)}{G(T-)} \right|\geq \eps\right)\leq \frac{5}{2} e^{-\eta}\,.
	\end{align*} 
	We obtained a bound for the first and second terms of the RHS of~\eqref{eq:prob_class_3terms}. The bound for the third term is given in~\eqref{eq:prob_omegan}, which concludes the proof.
	
	The following result, which will be needed later, can be obtained using the same arguments applied to the setting of Theorem~\ref{thm:hoeffding}.
\begin{align}\label{eq:bound_functions_thm2}
P\left(\frac{\sn\Htau\Gtauh}{M}\sup_{f\in\F}\left|\hat{\mu}_n(f)-\mu(f) \right|\geq 3\sqrt{\frac{\eta}{2}}+2D_o\right)\leq 5e^{-\eta}\,. 
\end{align}
	

\subsection{Proof of Theorem~\ref{thm:Bernstein}}\label{subsec:Bernstein}
Recall that by~\eqref{eq:diff_mun_mu},
\begin{align*}
&  \sn (\hat{\mu}_n(f)-\mu(f) )=A_n(f)-B_n(f)+C_n(f)\,,
\end{align*}
where
\begin{align*}
A_n(f)&=n^{-1/2}\sum_{i=1}^n \left(f(T_i,Z_i)-\mu(f) \right)\\
B_n(f)&=n^{-1/2}\sum_{i=1}^n\int_0^\tau \left(f(s,Z_i)-\frac{P[f(T,Z)\indi{T\geq s}]}{S(s)}\right)\frac{dM_i(s)}{G(s)}\\
C_n(f)=&n^{-1/2} \int_0^\tau\left[\frac{1}{n}\sum_{i=1}^{n}\left(\frac{\delta_j f(T_i,Z_i)}{\hatG(T_i)}\frac{\indi{T_i\geq s}}{\hat{S}(s-)}\right)- \frac{P[f(T,Z)\indi{T\geq s}]}{S(s-)}\right]\frac{dM^C(s)}{G(s)}\,.
\end{align*}
We would like to bound the expressions $A_n(f)$, $B_n(f)$, and $C_n(f)$. For $A_n(f)$, recall that $|f(t,z)|<M$. For $B_n(f)$,  by~\eqref{eq:identity} and Assumption~(A\ref{as:no_simultanious_failure}) we have
\begin{align}\label{eq:dMtoRV}
n^{-\frac12}\int_0^\tau \frac{dM^C(s)}{G(s)}=n^{-\frac12}\sum_{1=1}^{n}\int_0^\tau \frac{dM_i^C(s)}{G(s)}=n^{-\frac12}\sum_{i=1}^{n}\left(1-\frac{\delta_i}{G(T_i-)}\right)\,.
\end{align}
Using conditional expectations and the fact that $1/G(T_i-)<1/\Gtau$ we obtain that
\begin{align*}
&E\left(1-\frac{\delta_i}{G(T_i-)}\right)=E\left\{\left.E\left(1-\frac{\delta_i}{G(T_i-)}\right|T_i\right)\right\}=0\,,
\\
&\left|1-\frac{\delta_i}{G(T_i-)}\right|\leq \frac{1}{\Gtau}\,.
\end{align*}
Hence, 
\begin{align*}
\left|\int_0^\tau \left(f(s,Z_i)-\frac{P[f(T,Z)\indi{T\geq s}]}{S(s)}\right)\frac{dM_i(s)}{G(s)}\right|\leq\frac{2M}{\Htau}\,.
\end{align*}
Using Bernstein's inequality~\eqref{eq:Bernstein},
\begin{align*}
P\left(|A_n(f)+B_n(f)|\geq \sqrt{2\sigma_f^2\eta}+\frac{2M\eta}{3\sn}+\frac{4M\eta}{3\Htau\sn} \right)\leq
P\left(|A_n(f)+B_n(f)|\geq \sqrt{2\sigma_f^2\eta}+\frac{2M\eta}{\Htau\sn} \right)\leq 2e^{-\eta}\,.
\end{align*}

We now bound $C_n$. Note that
\begin{align*}
|C_n(f)|\leq &n^{-1/2} \sup_{s\in[0,\tau]}\left|\frac{1}{n}\sum_{i=1}^{n}\left(\frac{\delta_j f(T_i,Z_i)}{\hatG(T_i)}\frac{\indi{T_i\geq s}}{\hat{S}(s-)}\right)- \frac{P[f(T,Z)\indi{T\geq s}]}{S(s-)}\right|\left|\int_0^\tau\frac{dM^C(s)}{G(s)}\right| \,.
\end{align*}

Define 
\begin{align*}
f_s(t,z)=f(t,z)\indi{t\geq s}\,.
\end{align*}
Using this notation,
\begin{align*}
&\left|\frac{1}{n}\sum_{i=1}^{n}\left(\frac{\delta_j f(T_i,Z_i)}{\hatG(T_i)}\frac{\indi{T_i\geq s}}{\hat{S}(s-)}\right)- \frac{P[f(T,Z)\indi{T\geq s}]}{S(s)}\right|
=\left| \frac{\hat\mu_n(f_s)}{\hat{S}(s-)}-\frac{\mu(f_s)}{S(s-)}\right|
\\
&
\leq\left|\frac{\hat\mu_n(f_s)}{\hat{S}(s-)}-\frac{\hat\mu_n(f_s)}{S(s-)}\right|+\left|\frac{\hat\mu_n(f_s)}{S(s-)}-\frac{\mu(f_s)}{S(s-)}\right|
\\
&
\leq  \frac{1}{S(s-)}\left(\frac{\left|\hat{S}(s-)-S(s-)\right|}{\Stauh}\frac{2M}{\Gtauh}+\left|\hat\mu_n(f_s)-\mu(f_s)\right|\right)
\\
&
\leq
\frac{1}{\Stau}\left(\frac{2M}{\Htauh}\|\hat{S}-S\|_{\infty}+\left|\hat\mu_n(f_s)-\mu(f_s)\right|\right)\,.
\end{align*}	

By~\eqref{eq:KM_exp_bound}, replacing the roles of $S$ and $G$,
\begin{align*}
&P\left(\frac{\sn}{\Stau}\frac{2M}{\Htauh}\|\hat{S}-S\|_{\infty}\geq\frac{2M}{\Htau\Htauh}\left(\sqrt{\frac{\eta}{2}}+\frac{D_o}{2}\right)\right)
=
P\left(\sn \Gtau\|\hat{S}-S\|_{\infty}\geq\sqrt{\frac{\eta}{2}}+\frac{D_o}{2}\right)
\leq\frac{5}{2}e^{-\eta}
\end{align*}

Define the class of functions $\F=\{f_s: s\in[0,\tau]\}$ and let $\F_1=\{g(t)=\indi{t \geq s}: s\in[0,\tau]\}$. Note that $\F_1$ is a class of monotonic functions and that $\F= f\cdot \F_1$. Using Theorem~9.24 and Lemma~9.25 of \citet{Kosorok08}, we conclude that $
\log\mathcal{N}_{[]}(\eps,\F,L_2(P))<\gamma/\eps$ for some $\gamma>0$. Hence, by~\eqref{eq:bound_functions_thm2},
\begin{align*}
&P\left( \sup_{f_s\in\mathcal{F}}\frac{\sn}  {\Stau}\left|\hat{\mu}_n(f_s)-\mu(f_s) \right|\geq \frac{M}{\Htau\Stau\Gtauh}\left( 3\sqrt{\frac{\eta}{2}} +2D_o\right)\right)
\\
&= P\left( \sup_{f_s\in\mathcal{F}}\frac{\sn  \Htau\Gtauh}{M}\left|\hat{\mu}_n(f_s)-\mu(f_s) \right|\geq 3\sqrt{\frac{\eta}{2}} +2D_o\right)\leq 5e^{-\eta}\,.
\end{align*}
Using~\eqref{eq:dMtoRV} and~\eqref{eq:Hoeffding}, we have
\begin{align*}
P\left(n^{-\frac12}\left|\int_0^\tau \frac{dM^C(s)}{G(s)}\right|\geq \sqrt{\frac{\eta}{2G_{\tau}^2}}\right)\leq 2e^{-\eta}\,.
\end{align*}

Note that for two random variables $A_1$ and $A_2$, and two constants $a_1>0$ and $a_2>0$, we have
\begin{align*}
P(A_1 A_2\geq a_1 a_2)\leq P(A_1 \geq a_1 \cup A_2\geq a_2 )\leq P(A_1 \geq a_1)+P(A_2\geq a_2)\,. 
\end{align*}
Therefore,
\begin{align*}
&	P\left(|C_n(f)|\geq \frac{M}{\sn\Htau\Gtauh}\frac{\Stau+\Stauh}{\Stau\Stauh}\left(3\sqrt{\frac{\eta}{2}}+2D_o\right)\sqrt{\frac{2\eta}{G_{\tau}^2}}\right)
\\
&\leq P\left(\frac{\sn}{\Stau}\left(\frac{2M}{\Htauh}\|\hat{S}-S\|_{\infty}+\left|\hat\mu_n(f_s)-\mu(f_s)\right|\right)
\geq \frac{M}{\Htau\Gtauh}\left(\frac{1}{\Stauh}+\frac{1}{\Stau}\right)\left(3\sqrt{\frac{\eta}{2}}+2D_o\right)\right)
\\
&\quad+ P\left(n^{-\frac12}\left|\int_0^\tau \frac{dM^C(s)}{G(s)}\right|\geq \sqrt{\frac{2\eta}{G_{\tau}^2}}\right)
\\
&\leq P\left(\frac{\sn}{\Stau}\frac{2M}{\Htauh}\|\hat{S}-S\|_{\infty}\geq\frac{2M}{\Htau\Htauh}\left(\sqrt{\frac{\eta}{2}}+\frac{D_o}{2}\right)\right)
\\
&\quad+P\left( \sup_{f_s\in\mathcal{F}}\frac{\sn}  {\Stau}\left|\hat{\mu}_n(f_s)-\mu(f_s) \right|\geq \frac{M}{\Htau\Stau\Gtauh}\left( 3\sqrt{\frac{\eta}{2}} +2D_o\right)\right)
+2e^{-\eta}\leq \frac{19}{2}e^{-\eta}\,.
\end{align*}
Summarizing,
\begin{align*}
&P\left( |\hat{\mu}_n(f)-\mu(f) |\geq  \sqrt{\frac{2\sigma_f^2\eta}{n}}+\frac{2M\eta}{\Htau n}+
\frac{M}{n}\frac{\Stau+\Stauh}{H_{\tau}^2\Htauh}\left(3\sqrt{\frac{\eta}{2}}+2D_o\right)\sqrt{2\eta}
\right)
\\
&\leq P\left(|A_n(f)+B_n(f)|\geq \sqrt{2\sigma_f^2\eta}+\frac{2M\eta}{\Htau\sn}\right)
+P\left(|C_n(f)|\geq  
\frac{M}{\sn\Htau\Gtauh}\frac{\Stau+\Stauh}{\Stau\Stauh}\left(3\sqrt{\frac{\eta}{2}}+2D_o\right)\sqrt{\frac{2\eta}{G_{\tau}^2}}\right)
\\
&\leq \frac{23}{2} e^{-\eta}\,.
\end{align*}

\subsection{Proof of Corollary~\ref{cor:Bernstein}}\label{subsec:proof_cor_Bernstein}
From the proof of Theorem~\ref{thm:Bernstein}, we can conclude that
\begin{align}\label{eq:bound_cor}
\begin{split}
&P\left( |\hat{\mu}_n(f)-\mu(f) |\geq  \sqrt{\frac{2\sigma_f^2\eta}{n}}+\frac{2M\eta}{\Htau n}+
\frac{2M}{nH_{\tau}^3}\left(3\sqrt{\frac{\eta}{2}}+2D_o+3\right)\sqrt{2\eta}
\right)
\\
&\leq P\left(\frac{\sn}{\Stau}\frac{2M}{\Htauh}\|\hat{S}-S\|_{\infty}\geq \frac{2M}{\Htau^2\Stau}\left(\sqrt{\frac{\eta}{2}}+\frac{D_o}{2}+2\right)\right)
\\
&\quad+P\left( \sup_{f_s\in\mathcal{F}}\frac{\sn}  {\Stau}\left|\hat{\mu}_n(f_s)-\mu(f_s) \right|\geq \frac{M}{\Htau^2\Stau}\left( 3\sqrt{\frac{\eta}{2}} +2D_o+2\right)\right)+4 e^{-\eta}\,,
\end{split}
\end{align}
where the $4 e^{-\eta}$ term is obtained from the bounds on $|A_n(f)+B_n(f)|$ and $n^{-1/2}\left|\int_0^\tau dM^C(s)/G(s)\right|$.
We start by bounding the first expression on the RHS of~\eqref{eq:bound_cor}. For $n\geq 2$, define  $\Xi_n=\{(1+ 1/(\sn-1))\Htauh\geq \Htau\}$\,, and note that by the multiplicative Chernoff inequality~\citep[Eq.~7]{Hagerup1990}
\begin{align*}
&P(\Xi_n) =P\left( \left(1+ \frac{1}{\sn-1}\right)\Htauh\geq\Htau\right) =P\left(\Htauh \geq \left(1-\frac{1}{\sn}\right)\Htau\right)\leq 1- e^{-\Htau/(2n)}\,.
\end{align*}
Hence,
\begin{align}\label{eq:first_bound_cor}
\begin{split}
&P\left(\frac{\sn}{\Stau}\frac{2M}{\Htauh}\|\hat{S}-S\|_{\infty}\geq\frac{2M}{\Htau^2\Stau}\left(\sqrt{\frac{\eta}{2}}+\frac{D_o}{2}+2\right)\right)
\\
&\leq P\left(\left.\sn\Gtau\frac{\Htau}{\Htauh}\|\hat{S}-S\|_{\infty}\geq\sqrt{\frac{\eta}{2}}+\frac{D_o}{2}+2\right|\Xi_n\right)P(\Xi_n)+P(\Xi_n^C)
\\
&\leq 
P\left(\left.\sn\Gtau\left(1+\frac{1}{\sn-1}\right)\|\hat{S}-S\|_{\infty}\geq\sqrt{\frac{\eta}{2}}+\frac{D_o}{2}+2\right|\Xi_n\right)P(\Xi_n)+P(\Xi_n^C)
\\
&\leq
P\left(\sn\Gtau\|\hat{S}-S\|_{\infty}\geq\sqrt{\frac{\eta}{2}}+\frac{D_o}{2}+2-\frac{\sn}{\sn-1}\Gtau\|\hat{S}-S\|_{\infty}\right)+e^{-\Htau/(2n)}
\\
&\leq
P\left(\sn\Gtau\|\hat{S}-S\|_{\infty}\geq\sqrt{\frac{\eta}{2}}+\frac{D_o}{2}\right)+e^{-\Htau/(2n)}\leq \frac{5}{2}e^{-\eta}+e^{-\Htau/(2n)}\,.
\end{split}
\end{align}

For the second expression on the RHS of~\eqref{eq:bound_cor}, by Corollary~\ref{thm:hoeffding_cor},
\begin{align}\label{eq:sec_bound_cor}
\begin{split}
&P\left( \sup_{f_s\in\mathcal{F}}\frac{\sn}  {\Stau}\left|\hat{\mu}_n(f_s)-\mu(f_s) \right|\geq \frac{M}{\Htau^2\Stau}\left( 3\sqrt{\frac{\eta}{2}} +D_o+2\right)\right)
\\
&= P\left( \sup_{f_s\in\mathcal{F}}\frac{\sn H_{\tau}^2}  {M}\left|\hat{\mu}_n(f_s)-\mu(f_s) \right|\geq   3\sqrt{\frac{\eta}{2}} +D_o+2\right)\leq  5e^{-\eta}+e^{-\Htau/(3n)}\,.
\end{split}
\end{align}
The result follows by substituting~\eqref{eq:first_bound_cor} and~\eqref{eq:sec_bound_cor} in~\eqref{eq:bound_cor}.

\subsection{Proof of Lemma~\ref{lem:erm}}\label{subsec:proof_erm}
	Let $f_0=\argmin_{f\in\F}R_L(f) $, which exists since $R_L$ is continuous by Lemma~2.19 of \citet{SVR}, and $\F$ is compact. Note that
	\begin{align}\label{eq:diff_risk1}
	\begin{split}
	R_L(f_n^C)-R_L(f_0)\leq &R_L(f_n^C)-\hat{\mu}_n(L\circ f_n^C(T,Z) )  +\hat{\mu}_n(L\circ f_0)-R_L(f_0)
	\\
	\leq& 2\sup_{f\in\F}|\hat{\mu}_n(L\circ f)-\mu(L\circ f)|\,.
	\end{split}
	\end{align}	
	Since $\mathcal{N}(\F,\|\cdot \|_{\infty},\eps)<\infty$, there is a finite set of functions $\F_\eps\subset F$ such that for every $f\in\F$, there is a function $f^*\in\F_\eps$ for which $\|f-f^*\|\leq \eps$. Write 
	\begin{align}\label{eq:diff_risk2}
	\begin{split}
	\left|\hat{\mu}_n(L\circ f)-\mu(L\circ f)\right|
	\leq& \left|\hat{\mu}_n(L\circ f)-\hat{\mu}_n(L\circ f^*)\right|+
	\left|\hat{\mu}_n(L\circ f^*)-\mu(L\circ f^*)\right|+\left|\mu(L\circ f^*)-\mu(L\circ f)\right|\\
	\leq&  \frac{L_M\eps}{\Gtauh}+\left|\hat{\mu}_n(L\circ f^*)-\mu(L\circ f^*)\right|+L_M\eps\,,
	\end{split}
	\end{align}
	where we used the locally Lipschitz continuity part of~\eqref{eq:Lipchitz}. Substituting~\eqref{eq:diff_risk2} in~\eqref{eq:diff_risk1}, we obtain
	\begin{align*}
	R_L(f_n^C)-R_L(f_0)\leq \frac{2L_M\eps}{\Gtauh}+2\sup_{f^*\in\F_\eps}\left|\hat{\mu}_n(L\circ f^*)-\mu(L\circ f^*)\right|+2L_M\eps\,.
	\end{align*}
	Recall that by~\eqref{eq:Lipchitz}, $L\circ f\leq B$. Using the union bound~\eqref{eq:union_bound}, 
	\begin{align*}
	&P\left(\frac{\Htauh{\hat G_{\tau}^2}}{B}\left(R_L(f_n^C)-R_L(f_0)\right)\geq\frac{8\sqrt{2\eta^*}+6D_o}{\sn}+\frac{4\eps L_M}{B}\right)
	\\
	&\leq P\left(\frac{\Htauh{\hat G_{\tau}^2}}{B}\left(\frac{2L_M\eps}{\Gtauh}+2\sup_{f^*\in\F_\eps}\left|\hat{\mu}_n(L\circ f^*)-\mu(L\circ f^*)\right|+2L_M\eps\right)\geq\frac{8\sqrt{2\eta^*}+6D_o}{\sn}+\frac{4\eps L_M}{B}\right)
	\\
	&\leq
	P\left(\sup_{f\in\F_{\eps}}\frac{\Htauh{\hat G_{\tau}^2}}{B}\left|\hat{\mu}_n(L\circ f^*)-\mu(L\circ f^*)\right|
	\geq \frac{4\sqrt{2\eta^*}+3D_o}{\sn}\right)\leq \frac{11}{2}\mathcal{N}(\F,\|\cdot \|_{\infty},\eps) e^{-\eta^*}\,.
	\end{align*}
	The result follows by substituting $\eta=\eta^*-\log\mathcal{N}(\F,\|\cdot \|_{\infty},\eps)$.

\bibliographystyle{plainnat}

\begin{thebibliography}{30}
	\providecommand{\natexlab}[1]{#1}
	\providecommand{\url}[1]{\texttt{#1}}
	\expandafter\ifx\csname urlstyle\endcsname\relax
	\providecommand{\doi}[1]{doi: #1}\else
	\providecommand{\doi}{doi: \begingroup \urlstyle{rm}\Url}\fi
	
	\bibitem[Bang and Tsiatis(2000)]{Bang2000}
	H.~Bang and A.~A. Tsiatis.
	\newblock Estimating medical costs with censored data.
	\newblock \emph{Biometrika}, 87\penalty0 (2):\penalty0 329--343, 2000.
	
	\bibitem[Biganzoli et~al.(1998)Biganzoli, Boracchi, Mariani, and
	Marubini]{NN98}
	E.~Biganzoli, P.~Boracchi, L.~Mariani, and E.~Marubini.
	\newblock Feed forward neural networks for the analysis of censored survival
	data: \textsc{A} partial logistic regression approach.
	\newblock \emph{Statist. Med.}, 17\penalty0 (10):\penalty0 1169--1186, 1998.
	
	\bibitem[Bitouz{\'e} et~al.(1999)Bitouz{\'e}, Laurent, and Massart]{Bitouze99}
	D.~Bitouz{\'e}, B.~Laurent, and P.~Massart.
	\newblock A {D}voretzky-{K}iefer-{W}olfowitz type inequality for the
	{K}aplan-{M}eier estimator.
	\newblock \emph{Ann. Inst. H. Poincar\'e Probab. Statist.}, 35\penalty0
	(6):\penalty0 735--763, 1999.
	
	\bibitem[Boucheron et~al.(2004)Boucheron, Lugosi, and
	Bousquet]{Boucheron2004Concentration}
	S.~Boucheron, G.~Lugosi, and O.~Bousquet.
	\newblock Concentration inequalities.
	\newblock In O.~Bousquet, U.~von Luxburg, and Gunnar R\"{a}tsch, editors,
	\emph{Advanced Lectures on Machine Learning}, volume 3176 of \emph{Lecture
		Notes in Computer Science}, pages 208--240. Springer Berlin / Heidelberg,
	2004.
	
	\bibitem[Chung and Lu(2006)]{Chung06}
	F.~Chung and L.~Lu.
	\newblock {Concentration inequalities and martingale inequalities: a survey}.
	\newblock \emph{Internet Mathematics}, 3\penalty0 (1):\penalty0 79--127, 2006.
	
	\bibitem[Dabrowska(1987)]{Dabrowska87}
	D.~M. Dabrowska.
	\newblock {Non-Parametric} regression with censored survival time data.
	\newblock \emph{Scandinavian Journal of Statistics}, 14\penalty0 (3):\penalty0
	181--197, 1987.
	
	\bibitem[Dabrowska(1989)]{Dabrowska89}
	D.~M. Dabrowska.
	\newblock Uniform consistency of the kernel conditional {Kaplan-Meier}
	estimate.
	\newblock \emph{The Annals of Statistics}, 17\penalty0 (3):\penalty0
	1157--1167, 1989.
	
	\bibitem[Efron(1967)]{Efron1967}
	B.~Efron.
	\newblock The two sample problem with censored data.
	\newblock In \emph{Proceedings of the Fifth Berkeley Symposium on Mathematical
		Statistics and Probability, Volume 4: Biology and Problems of Health}, pages
	831--853. University of California Press, 1967.
	
	\bibitem[Eleuteri and Taktak(2012)]{Eleurti_2012}
	A.~Eleuteri and A.~F.~G. Taktak.
	\newblock Support vector machines for survival regression.
	\newblock In E.~Biganzoli, A.~Vellido, F.~Ambrogi, and R.~Tagliaferri, editors,
	\emph{Computational Intelligence Methods for Bioinformatics and
		Biostatistics}, volume 7548, pages 176--189. Springer Berlin Heidelberg,
	2012.
	
	\bibitem[Fleming and Harrington(1991)]{FH}
	T.~R. Fleming and D.~P. Harrington.
	\newblock \emph{Counting processes and survival analysis}.
	\newblock Wiley, 1991.
	
	\bibitem[Gill(1980)]{gill1980censoring}
	R.~D. Gill.
	\newblock Censoring and stochastic integrals.
	\newblock \emph{Statistica Neerlandica}, 34\penalty0 (2):\penalty0 124--124,
	1980.
	
	\bibitem[Goldberg and Kosorok(2012)]{GK_CQL11}
	Y.~Goldberg and M.~R. Kosorok.
	\newblock Q-learning with censored data.
	\newblock \emph{The Annals of Statistics}, 40\penalty0 (1):\penalty0 529--560,
	2012.
	
	\bibitem[Goldberg et~al.(2017)Goldberg, Kosorok, et~al.]{GK_SVR}
	Yair Goldberg, Michael~R Kosorok, et~al.
	\newblock Support vector regression for right censored data.
	\newblock \emph{Electronic Journal of Statistics}, 11\penalty0 (1):\penalty0
	532--569, 2017.
	
	\bibitem[Hagerup and R\"{u}b(1990)]{Hagerup1990}
	T.~Hagerup and C.~R\"{u}b.
	\newblock A guided tour of chernoff bounds.
	\newblock \emph{Information Processing Letters}, 33\penalty0 (6):\penalty0
	305--308, 1990.
	
	\bibitem[Ishwaran and Kogalur(2010)]{Ishwaran2010}
	H.~Ishwaran and U.~B. Kogalur.
	\newblock Consistency of random survival forests.
	\newblock \emph{Statistics \& Probability Letters}, 80\penalty0
	(13-14):\penalty0 1056--1064, 2010.
	
	\bibitem[Johnson et~al.(2004)Johnson, Lin, Marron, Ahn, Parker, and
	Perou]{Johnson04}
	B.~A. Johnson, D.~Y. Lin, J.~S. Marron, J.~Ahn, J.~Parker, and C.~M. Perou.
	\newblock Threshhold analyses for inference in high dimension low sample size
	datasets with censored outcomes.
	\newblock Unpublished manuscript, 2004.
	
	\bibitem[Koltchinskii(2011)]{Koltchinskii2011}
	Vladimir Koltchinskii.
	\newblock \emph{{Oracle inequalities in empirical risk minimization and sparse
			recovery problems. \'Ecole d'\'Et\'e de Probabilit\'es de Saint-Flour
			XXXVIII-2008.}}
	\newblock Springer, Berlin, 2011.
	
	\bibitem[Kosorok(2008)]{Kosorok08}
	M.~R. Kosorok.
	\newblock \emph{Introduction to Empirical Processes and Semiparametric
		Inference}.
	\newblock Springer, New York, 2008.
	
	\bibitem[Maurer and Pontil(2009)]{Maurer2009EmpiricalBernstein}
	A.~Maurer and M.~Pontil.
	\newblock Empirical {B}ernstein bounds and sample-variance penalization.
	\newblock In \emph{COLT 2009 - The 22nd Conference on Learning Theory}, 2009.
	
	\bibitem[Ripley and Ripley(2001)]{Ripley}
	B.~D. Ripley and R.~M. Ripley.
	\newblock Neural networks as statistical methods in survival analysis.
	\newblock In Ri. Dybowski and V.~Gant, editors, \emph{Clinical Applications of
		Artificial Neural Networks}, pages 237--255. Cambridge University Press,
	2001.
	
	\bibitem[Robins and Rotnitzky(1992)]{robins1992recovery}
	J.~M. Robins and A.~Rotnitzky.
	\newblock Recovery of information and adjustment for dependent censoring using
	surrogate markers.
	\newblock In \emph{AIDS Epidemiology}, pages 297--331. Springer, 1992.
	
	\bibitem[Robins et~al.(1994)Robins, Rotnitzky, and Zhao]{Robins94}
	J.~M. Robins, A.~Rotnitzky, and L.~P. Zhao.
	\newblock Estimation of regression coefficients when some regressors are not
	always observed.
	\newblock \emph{Journal of the American Statistical Association}, 89\penalty0
	(427):\penalty0 846--866, 1994.
	
	\bibitem[Satten and Datta(2001)]{Satten01}
	G.~A. Satten and S.~Datta.
	\newblock The \textsc{K}aplan-\textsc{M}eier estimator as an
	inverse-probability-of-censoring weighted average.
	\newblock \emph{The American Statistician}, 55\penalty0 (3):\penalty0 207--210,
	2001.
	
	\bibitem[Schick et~al.(1988)Schick, Susarla, and Koul]{Schick1988}
	A.~Schick, V.~Susarla, and H.~Koul.
	\newblock Efficient estimation of functionals with censored data.
	\newblock \emph{Statistics \& Risk Modeling}, 6\penalty0 (4):\penalty0
	349--360, 1988.
	
	\bibitem[Shim and Hwang(2009)]{SVQR09}
	J.~Shim and C.~Hwang.
	\newblock Support vector censored quantile regression under random censoring.
	\newblock \emph{Computational Statistics \& Data Analysis}, 53\penalty0
	(4):\penalty0 912--919, 2009.
	
	\bibitem[Shivaswamy et~al.(2007)Shivaswamy, Wei, and Jansche]{SVR07}
	P.~K. Shivaswamy, C.~Wei, and M.~Jansche.
	\newblock A support vector approach to censored targets.
	\newblock In \emph{Seventh IEEE International Conference on Data Mining, 2007},
	pages 655--660, 2007.
	
	\bibitem[Steinwart and Chirstmann(2008)]{SVR}
	I.~Steinwart and A.~Chirstmann.
	\newblock \emph{Support Vector Machines}.
	\newblock Springer, 2008.
	
	\bibitem[Vapnik(1999)]{Vapnik99}
	V.~Vapnik.
	\newblock \emph{The Nature of Statistical Learning Theory}.
	\newblock Springer, 2nd edition, 1999.
	
	\bibitem[Wellner(2007)]{Wellner07}
	J.~Wellner.
	\newblock On an exponential bound for the {Kaplan–Meier} estimator.
	\newblock \emph{Lifetime Data Analysis}, 13\penalty0 (4):\penalty0 481--496,
	2007.
	
	\bibitem[Zhao and Tsiatis(1997)]{zhao_consistent_1997}
	H.~Zhao and A.~A. Tsiatis.
	\newblock A consistent estimator for the distribution of quality adjusted
	survival time.
	\newblock \emph{Biometrika}, 84\penalty0 (2):\penalty0 339--348, 1997.
	
\end{thebibliography}

\end{document}